\documentclass[12pt,a4paper]{amsart}

\usepackage[utf8]{inputenc}
\usepackage[english]{babel}

\usepackage{latexsym}
\usepackage{amsfonts}
\usepackage{amsmath}
\usepackage{amssymb}
\usepackage{graphicx}
\usepackage{color}
\usepackage{enumitem}

\usepackage{amsthm}
\usepackage{pstricks}
\usepackage{csquotes}
\usepackage{fullpage}
\usepackage{url}
\usepackage[ocgcolorlinks, linkcolor=red]{hyperref}

\usepackage{algorithm}
\usepackage{algcompatible}

\usepackage[numbers]{natbib}

\newcommand{\R}{\mathbb{R}}

\newcommand{\C}{\mathbb{C}}

\newcommand{\ov}[1]{\overline{#1}}

\newcommand\supp{\operatorname{supp}} 
 
\newcommand\spec{\operatorname{Spec}}

\newcommand\p{\partial}

\newcommand\osupp{\operatorname{osupp}}

\newtheorem{thm}{Theorem}[section]

\newtheorem{lem}[thm]{Lemma}
\newtheorem{prop}[thm]{Proposition}

\newtheorem{rem}[thm]{Remark}

\theoremstyle{definition}

\numberwithin{equation}{section}

\begin{document}

\title[]{Experimental detection of inclusions via a monotonicity method for the time-harmonic elastic wave equation}

\date{}

\author[]{Sarah Eberle-Blick}
\address{Sarah Eberle-Blick, Mathematical Institute for Machine Learning and Data Science,
Catholic University of Eichst\"att-Ingolstadt, Germany}
\email{sarah.eberle-blick@ku.de}

\author[]{Jochen Moll}
\address{Jochen Moll, Department of Mechanical Engineering, University of Siegen, Germany}
\email{jochen.moll@uni-siegen.de}

\begin{abstract}
We are concerned with the reconstruction of inclusions in elastic bodies based on measurements from a laboratory experiment. In doing so, we solve the inverse problem of the time-harmonic elastic wave equation, in contrast to the stationary wave equation and the corresponding lab experiment proposed earlier in \cite{EM21}. The investigation of the harmonic problem leads to a better reconstruction compared to the stationary one. Since we deal with real measurement data, we have to take into account, that those measurements always include measurement errors, so that we have to handle noisy data. Thus, we consider the linearized monotonicity method for noisy data and introduce a modified version of this method. Based on this, we reconstruct the inclusions numerically.
\end{abstract}

\maketitle

\section{Introduction} 
\noindent
We want to find inclusions in elastic bodies and reconstruct their shape. Mathematically, this is an inverse problem, meaning that only boundary measurements are known. Specifically, a time-harmonic elastic wave equation is considered, which provides the necessary information. Inverse problems of elasticity have many applications such as material analysis, geophysical and medical examinations, etc. (see e.g. \cite{CK98,GK08,CC06, Pot06} and the references therein), but we are focusing on material analysis.
\\
\\
In this work, we perform a lab experiment in order to obtain the required data and apply a monotonicity method for the reconstruction of the inclusions based on real data.
\\
\\
The paper is structured as follows: 
\\
We start with a short overview of the considered problem. Then we introduce the experimental setup and give the background of the applied numerical method, i.e. the linearized monotonicity method for the elasto-oscillatory wave equation. In more detail, we modify this method for the specific requirements of our experimental data. Finally, we present our numerically simulations for reconstructing the inclusions.

\section{Problem Statement}
\noindent
We follow the description and introduction of the problem from \cite{EP24}. Let us introduce the problem of the time-harmonic elastic wave equation as follows. We consider a body $\Omega\subset \mathbb{R}^3$ with Neumann boundary $\Gamma_N$ and Dirichlet boundary $\Gamma_D$ such that
$$
\Gamma_N, \Gamma_D \subset \p\Omega \text{ are open }, \qquad \Gamma_N \neq \emptyset, \qquad
\p \Omega = \ov{\Gamma}_N \cup \ov{\Gamma}_D
$$ 
and $\nu$ is the outward pointing unit normal vector to the boundary $\p \Omega$.
\\
\\
Further on, the specific material properties are described via the Lamé parameters $\lambda,\mu \in L^\infty_+(\Omega)$, which determine the elastic  properties of the material, $\rho \in L^\infty_+(\Omega)$ is the density of the material,
and $\omega \neq 0$ a non-resonance angular frequency of the oscillation. 
\\
\\
All in all, the displacement field $u : \Omega \to \R^3$, $u \in H^1(\Omega)^3$ of the solid body $\Omega$ fulfills the boundary value problem
\begin{align}  \label{eq_bvp1}
\begin{cases}
\nabla \cdot (\C\,  \hat \nabla u )  + \omega^2\rho u &= 0, \quad\text{in}\,\,\Omega,\\
\hspace{1.8cm}(\C\,  \hat \nabla u ) \nu  &= g, \quad\text{on}\,\,\Gamma_N, \\
\hphantom{\hspace{1.8cm}(\C\,  \hat \nabla u ) } u  &= 0, \quad\text{on}\,\,\Gamma_D. 
\end{cases}
\end{align}
The involved differential operators are given as 
$\hat \nabla u  = \frac{1}{2}(\nabla u + (\nabla u)^T)$, which is the symmetrization of the Jacobian
or the strain tensor, 
and $\C$, which is the 4th order tensor defined by
\begin{equation} \label{eq_Cdef}
\begin{aligned}
(\C A)_{ij} = 2\mu A_{ij} + \lambda \operatorname{tr}(A) \delta_{ij},
\quad \text{ where } A \in \R^{3 \times 3},
\end{aligned}
\end{equation}
and $\delta_{ij}$ is the Kronecker delta. 
\\
\\
Finally, the Neumann boundary condition acts as the source
of the oscillation via the vector field $g \in L^2(\Gamma_N)^3$ and this boundary condition $g$ specifies the traction on the surface $\p \Omega$.
\\
\\
We assume that $\omega \in \R$ is not a resonance frequency, so that problem \eqref{eq_bvp1} has a unique solution for a given boundary condition $g\in L^2(\Gamma_N)^3$. Based on this, we can thus define the Neumann-to-Dirichlet map
$\Lambda:L^2(\Gamma_N)^3 \to L^2(\Gamma_N)^3$, as
\begin{equation} \label{eq_ND_map}
\begin{aligned}
\Lambda: g \mapsto u|_{\Gamma_N}. 
\end{aligned}
\end{equation}
Thus $\Lambda$  maps the traction to the displacement $u|_{\Gamma_N}$ on the boundary.
\\
\\
\noindent
In this paper we are concerned with the reconstruction of inclusions with real data (see, e.g. \cite{EH21} for the stationary case) based on a monotonicity method
for elasticity especially for the time-harmonic problem as in \cite{EP24}, \cite{EP24a} and \cite{E25}.
\\
\\
We are more specifically interested in determining the shape of perturbations of the material parameters
$\lambda, \mu$ and $\rho$, in an otherwise homogeneous  background material characterized by the constant
material parameters $\lambda_0, \mu_0, \rho_0 > 0$. We would  ideally want to reconstruct the sets
$$
\supp(\lambda-\lambda_0), \quad
\supp(\mu-\mu_0),  \quad
\supp(\rho-\rho_0)
$$
from the Neumann-to-Dirichlet map $\Lambda$. 
\\
\\
We base our simulations on the linearized monotonicity method in order to reconstruct the set $\osupp(D)$, where 
$$
D = \supp(\lambda-\lambda_0) \cup \supp(\mu-\mu_0) \cup \supp(\rho- \rho_0),
$$
with $\osupp(D)$ as outer support of $D$ as we did in \cite{EP24a} with a linearized method. 
\\
\\
The linearized method presented here is thus an improvement of the method in
\cite{EP24a} in terms of computation time. The linearized method is drastically faster to compute,
which is the original motivation for considering the linearized method, and is therefore of great interest
from a computational point of view, see \cite{HU13,HS10}.
The reason is, that in comparison to the standard method (see, e.g. \cite{EP24}), the linearized method does not require calculating forward problems for different test inclusions or experimental measurements of test subjects with different test inclusions.

\section{Experimental Setup} \label{sec_experiment}
\noindent
These experiments build on those already carried out by the authors for the stationary wave equation (see \cite{EM21}), so that the already gained know-how can be used for the new laboratory experiments. Thus, the setup of the experiment was further developed together with the cooperation partner of the Fraunhofer Institute for Structural Durability and System Reliability LBF who also assisted with setting up the equipment.
\\
\\
For a better comparison of the results from the static and oscillatory case, we tested the same samples in a similar measurement setup. This means that Makrolon plates with various aluminum inclusions (see \cite{EM21}) are tested. 
\\
\\
In the case of the elasto-oscillatory problem, i.e. for time-harmonic Neumann data, a "sinusoidal" force (from $-400$N to $+400$N) is attached to the edge of the Makrolon plate and the resulting displacement is measured and transferred to the frequency range via Fourier transformation. Even if this sample is flat and would therefore justify a 2D calculation, all simulations and experiments are carried out in a 3D setup. 
\\
\\
\begin{rem}
The Neumann data we deal with are vibrations. Hence, we are only interested in the displacements resulting from these vibrations and not in solid-borne sound velocities or the resulting wavelengths. The frequency range (20Hz - 27Hz, 40Hz - 45Hz and 55Hz - 57Hz) was chosen for technical reasons. It must be ensured that the optical measuring sensor can detect the displacements. Further on, frequencies resulting in resonance phenomena are excluded. We searched for the resonance frequencies via experiments. It should be noted that under the excitation frequency, the shape is reconstructed via the difference of the displacement of the body with an inclusion and the displacement of the body without inclusion and not via sampling of the probe with waves. Further, if the frequency tends to zero, we find ourselves in the static case which provides the shape reconstruction as described in \cite{EH21}, \cite{EH22}, \cite{EH23}.
\end{rem}
\noindent
\\
The resulting data is then used to detect and reconstruct the inclusions using linearized monotonicity methods. We want to mention, that they cover a larger frequency range and more general parameters than the specific laboratory experiment. Contrary to the laboratory experiment, even higher frequencies can be investigated in the numerical simulations based on artificial data, but in the laboratory experiment, only the specified frequencies are investigated for the reasons mentioned above.
\\
\\
The setup is given as follows:
\\
\\
We applied a force via a sine-sweep between $-400$ and $400$N on two boundary patches (see Figure \ref{setup_overview}: item 3) which are fixed at the Makrolon plate (see Figure \ref{setup_overview}: item 5) and connected with a TV 50101 shaker (TIRA, Germany) (see Figure \ref{setup_overview}: item 1). The Makrolon plate is fixed as shown in  Figure \ref{setup_overview}: item 4. The resulting displacement is measured with an OM70 high performance optical distance sensor (Baumer GmbH, Germany) (see Figure \ref{setup_overview}: item 2).

\begin{figure}[H]
\centering 
\includegraphics[width=0.8\textwidth]{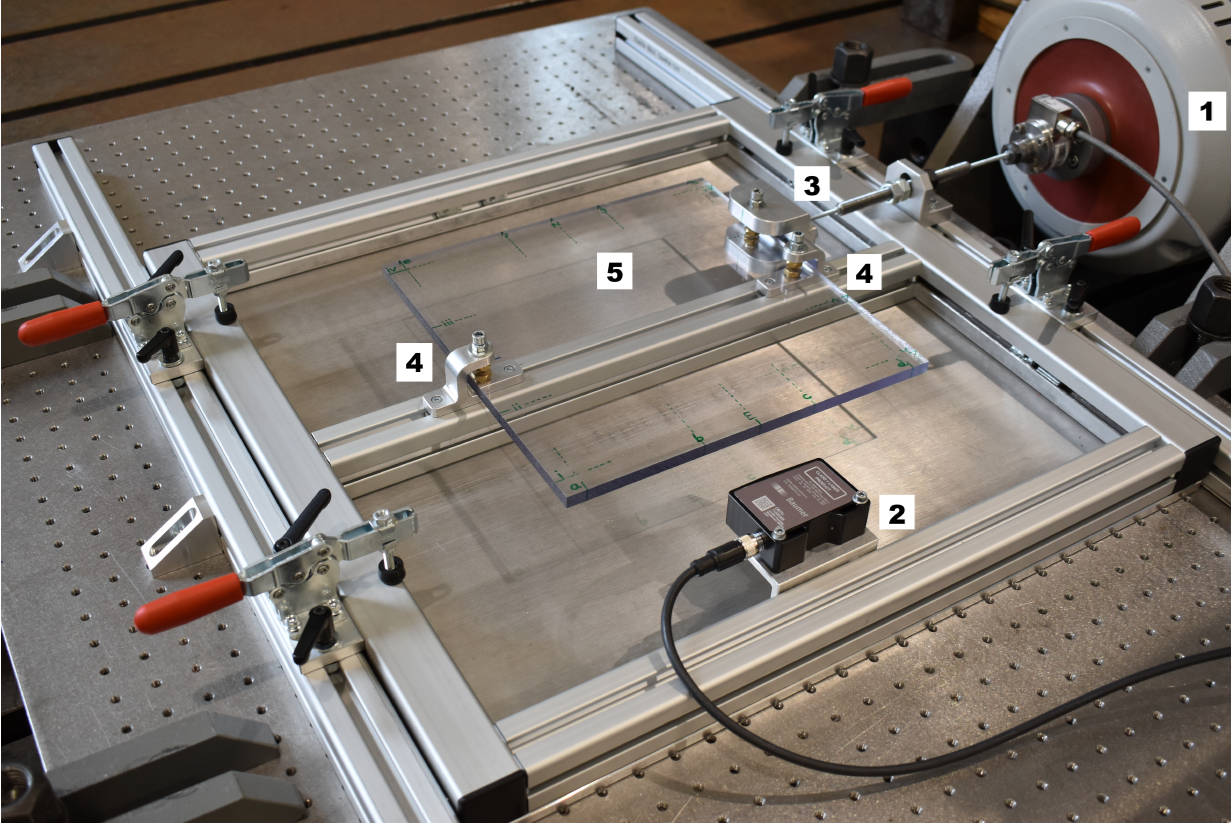}
\caption{Setup overview: 1: shaker, 2: optical sensor, 3: connection shaker and Makrolon plate, 4:  two mounting points, where the Makrolon plate is fixed at the frame, 5: Makrolon plate\label{setup_overview}}
\end{figure}
\noindent
\\
Finally, we take a look at the schematic setting of the Makrolon plate:
\begin{figure}[H]
\centering 
\includegraphics[width=0.57\textwidth]{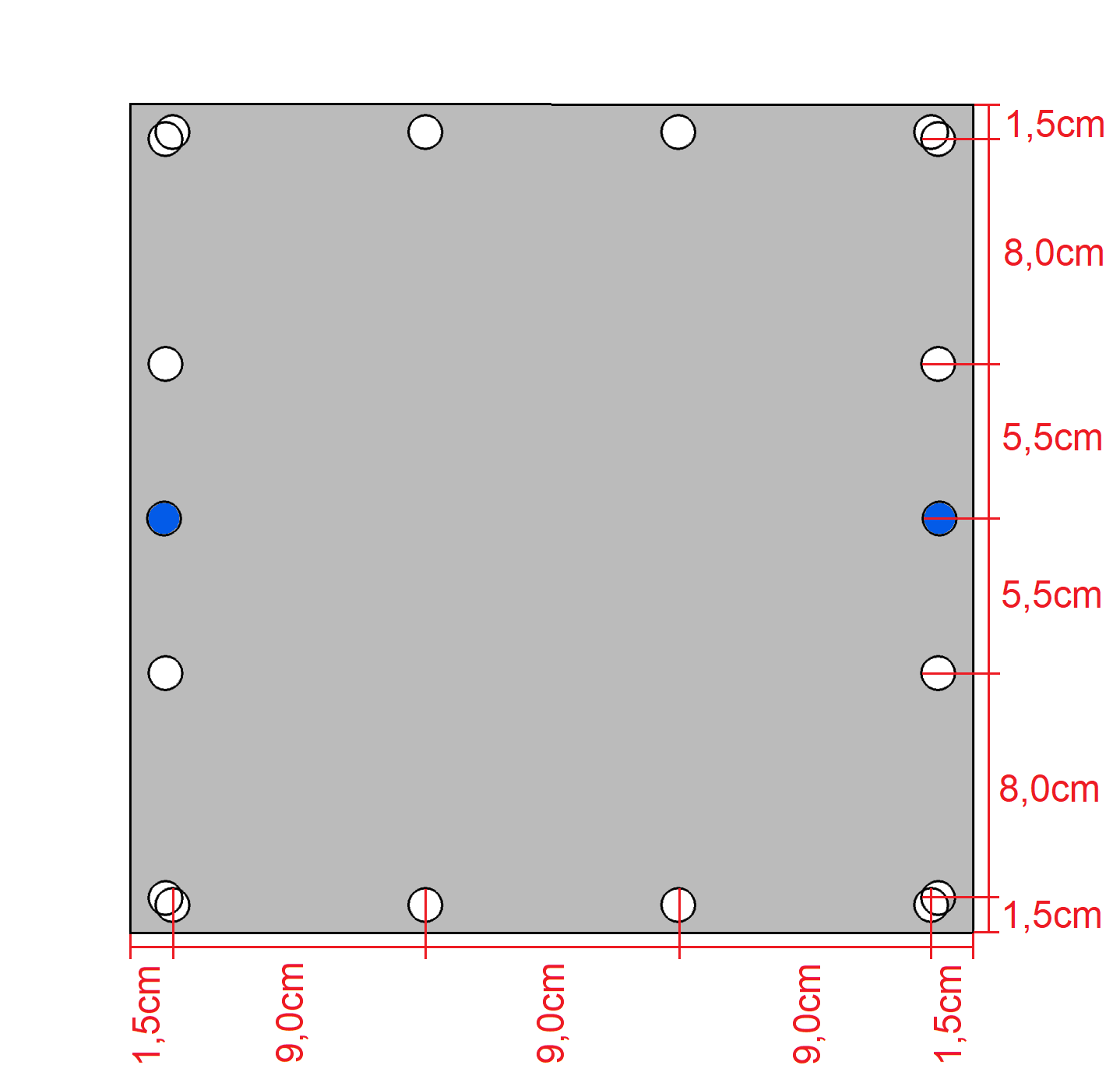}
\caption{Schematic setting: Makrolon plate with positions of the mounting (circles): Dirichlet conditions in blue and Neumann conditions in white.}
\end{figure}

\section{Linearized monotonicity test} \label{sec_mono_shape}

\noindent
In this section we will introduce the linearized monotonicity tests as considered in \cite{EP24a}. First of all, we give the required background:
Our notations are to a large extent the same notations as in \cite{EP24}.
\\
\\
Let the space $H^k(\Omega)$ denote the $L^2(\Omega)$ based Sobolev space with $k$ weak derivatives. In addition, we define the space
$$	
L^\infty_+(\Omega) := \big\{ f \in L^\infty(\Omega) \;:\; \operatorname{essinf}_\Omega f > 0 \big\}.
$$
\noindent
\\
Next, we state the bilinear form $B$ for problem \eqref{eq_bvp1} as
\begin{align}  \label{eq_weak}
B(u,v)  := -\int_\Omega 
2 \mu \hat \nabla u :\hat \nabla v + \lambda \nabla \cdot u \nabla \cdot v - \omega^2\rho u\cdot v\,dx, 
\end{align}
for all $u,v \in H^1(\Omega)^3$.
\\
\\
A weak solution to \eqref{eq_bvp1} is defined as a $u \in H^1(\Omega)^3$,
for which $u|_{\Gamma_D} = 0$.
\\
\\
If, we assume that the material parameter $\lambda, \mu, \rho$ are regular and $u$ solves \eqref{eq_bvp1} with $g$, and
$v$ solves \eqref{eq_bvp1} with $h$, applying integration by parts leads to
\begin{align}  \label{eq_NDmap}
B(u,v)  = -\int_{ \Gamma_N} g \cdot v \,dS = -( g, \Lambda h)_{L^2(\Gamma_N)^3}. 
\end{align}
\\
\\
In \cite{EP24} (see Corollary 3.4) we proved the existence and uniqueness of a weak solution to \eqref{eq_bvp1}, when $\omega$ is not a resonance frequency.
\\
\\
This is important, since from the existence and uniqueness of a weak solution to \eqref{eq_bvp1}, we can directly conclude the well-definition of 
the Neumann-to-Dirichlet map $\Lambda$ given by \eqref{eq_ND_map}. 
\\
\\
Hence, we use the abbreviations of the boundary condition in \eqref{eq_bvp1} if the boundary is clear from the context
\begin{align*} 
\gamma_{\mathbb{C},\Gamma} u  = (\C \, \hat \nabla u ) \nu |_{\Gamma} \quad\text{or}\quad \gamma_\mathbb{\C} u .
\end{align*}
\\
\\
Since we will reconstruct the shape of the unknown inclusions, we introduce the framework of an outer support.
\\
\\
The outer support (with respect to $\p\Omega$) of a measurable function $f: \Omega \to \R$   is defined as 
$$
\operatorname{osupp} (f) := \Omega \setminus \bigcup \, \big \{  U \subset \Omega \,:\, U  
\text{ is relatively open and connected to $\p \Omega$},
f|_U \equiv 0  \big\}.
$$
For more on this see \cite{HU13}.
It will be convenient to extend this definition to sets. We define the outer support of a measurable
set $D \subset \Omega$ 
(with respect to $\p\Omega$) as
\begin{align}  \label{eq_def_osupp}
\osupp (D) := \osupp(\chi_D)
\end{align}
where $\chi_D$ is the characteristic function of the set $D$. 

\medskip
\noindent
For conducting the test themselves, we use the following results from \cite{EP24} and \cite{EP24a}.
\\
\begin{lem}[see Lemma 4.1 from \cite{EP24}] \label{lem_monotonicity_ineq1}
Let $\mu_j,\lambda_j,\rho_j \in L^\infty_+(\Omega)$, for $j=1,2$ and $ \omega \neq 0$.
Let $u_j$ denote the solution to \eqref{eq_bvp1} where $\mu=\mu_j, \lambda= \lambda_j$ and $ \rho=\rho_j$,
with the boundary value  $g$. There exists a finite dimensional subspace $V\subset L^2( \Gamma_N)^3$, 
such that
\begin{align*}  
\big(  (\Lambda_2 - \Lambda_1)g, \,g  \big )_{L^2(\Gamma_N)^3} 
\geq
\int_\Omega 
2(\mu_1-\mu_2 ) |\hat \nabla u_1 |^2 + (\lambda_1 - \lambda_2 ) |\nabla \cdot u_1|^2  + \omega^2(\rho_2-\rho_1) |u_1|^2 \,dx, 
\end{align*}
when $g \in V^\perp$. 
\end{lem}
\noindent
\\
Before we take a look at the Fréchet derivative of the Neumann-to-Dirichlet
map, we need some estimates which are given in the following lemmata.
\begin{lem}[see Lemma 3.1 from \cite{EP24a}]\label{lem_ab_1}
Let $\mu_j,\lambda_j,\rho_j \in L^\infty_+(\Omega)$, for $j=1,2$ and $ \omega \neq 0$.
Let $u_j$ denote the solution to \eqref{eq_bvp1} where $\mu=\mu_j, \lambda= \lambda_j$ and $ \rho=\rho_j$,
with the boundary value $g$. Then
\begin{align*}
\left ((\Lambda_2  - \Lambda_1) g, g \right )_{L^2(\Gamma_{\textup{N}})^3}
\leq &\int_\Omega2(\mu_1-\mu_2)|\hat\nabla u_2|^2+(\lambda_1-\lambda_2)|\nabla\cdot u_2|^2+\omega^2(\rho_2-\rho_1)|u_2|^2dx \\
&+\omega^2\int_\Omega\rho_1|u_1-u_2|^2dx .
\end{align*}
\end{lem}
\noindent
\\
Finally, we take a look at the Fréchet derivative (as introduced in \cite{EP24a}), which is needed for the linearized monotonicity tests.

\begin{lem} [see Lemma 3.5 from \cite{EP24a}]\label{lem_Frechet}
Let $u_g$ and $u_f$ be the solution to (\ref{eq_bvp1}) for the boundary loads $g$ and $f$, respectively.
There exists a bounded linear operator $\Lambda^\prime_{\lambda,\mu,\rho}$ such that
\begin{align*} 
\lim_{\| h\|_{\infty} \to 0}
\dfrac{1}{\| h\|_{\infty}} 
\big\|\Lambda_{\lambda+h_\lambda,\mu+h_\mu,\rho+h_\rho}-\Lambda_{\lambda,\mu,\rho}
-\Lambda^\prime_{\lambda,\mu,\rho}[h_\lambda,h_\mu,h_\rho] \big\|_{*}
= 0
\end{align*} 
where $h=(h_\rho,h_\lambda,h_\mu) \in L^\infty(\Omega)^3$, $ \| \cdot \|_{*}$ is the operator norm, and
where the associated bilinear form of the Fréchet derivative is
\begin{align}\label{bilinear_Frechet}
&\left( \Lambda^\prime_{\lambda,\mu,\rho}[h_\lambda,h_\mu,h_\rho] g,f \right )_{L^2(\Gamma_{\textup{N}})^3}
=-\int_{\Omega} 2 h_\mu \hat{\nabla}u_g : \hat{\nabla} u_f + h_\lambda \nabla \cdot u_g \nabla \cdot u_f - \omega^2 h_\rho u_g \cdot u_f \,dx.
\end{align}
\end{lem}

\begin{lem} [see Lemma 3.6 from \cite{EP24a}]\label{lem_compact_self_adjoint}
The Frechét derivative $\Lambda_{\lambda,\mu,\rho}'$ is compact and self-adjoint.
\end{lem}
\noindent
\\
Now, we can continue with the linearized monotonicity methods themselves.
We will consider inhomogeneities in the material parameters of the following form.
We let $D_1, D_2, D_3 \Subset \Omega$, and assume that $\lambda,\mu, \rho \in L_+^\infty(\Omega)$ 
are such that
\begin{equation} \label{eq_lambdaMuRho}
\begin{aligned} 
\lambda &= \lambda_0 + \chi_{D_1} \psi_\lambda , \qquad \psi_\lambda \in L^\infty(\Omega),
\quad \psi_\lambda > \epsilon, \\
\mu &= \mu_0 + \chi_{D_2} \psi_\mu, \qquad \psi_\mu \in L^\infty(\Omega),\quad \psi_\mu > \epsilon, \\
\rho \,\,&{\color{black} = \rho_0 - \chi_{D_3} \psi_\rho, \qquad \psi_\rho \in L^\infty(\Omega),\quad  \rho_0-\epsilon > \psi_\rho > \epsilon}, 
\end{aligned}
\end{equation}
where the constants fulfill $\lambda_0,\mu_0,\rho_0 > 0$ and $\epsilon > 0$.
The coefficients $\lambda,\mu$ and $\rho$  model inhomogeneities in an otherwise homogeneous background medium
determined by $\lambda_0,\mu_0$ and $\rho_0$. 
\noindent
\\
\\
We continue with the linearized monotonicity tests, where we distinguish between the formulation for exact and noisy data. We base our considerations on Theorem \ref{thm_Lin_inclusionDetection} from \cite{EP24a} and introduce an essential modification in Theorem \ref{thm_Lin_inclusionDetection_Mod} for realistic parameter combination.

\begin{thm}[see Theroem 6.1 from \cite{EP24a}]\label{thm_Lin_inclusionDetection}
Let $D := D_1 \cup D_2 \cup D_3$, where the sets are as in \eqref{eq_lambdaMuRho} and
$B \subset \Omega$  and $\alpha_j > 0$, 
and set $\alpha:=(\alpha_1,\alpha_2,\alpha_3)$. Let
$$
\mathcal{M}:= \# \big\{\sigma \in \spec(\Lambda_0 + \Lambda'_0[\alpha_1,\alpha_2,-\alpha_3] - \Lambda)
\;:\; \sigma < 0 \big\},
$$
where $\Lambda_0$ and $\Lambda$ are the NtD-maps for the coefficients $\lambda_0,\mu_0,\rho_0$ and $\lambda,\mu,\rho$ respectively,
and where $\Lambda'_0[\alpha_1,\alpha_2,-\alpha_3] := \Lambda'_0[\alpha_1\chi_B,\alpha_2\chi_B,-\alpha_3\chi_B]$. 
There exists a $\gamma_0 > 0$ such that the following holds: \\
\begin{enumerate}

\item 
Assume that  $B \subset D_j$, for $j \in J$, for some $J \subset\{1,2,3\}$.  
Then for all $\alpha_j$ with $\alpha_j \leq \gamma_0$, $j \in J$, and $\alpha_j = 0$, 
$j \notin J$, we have that  $\mathcal{M}< \infty$.
\\
\item
If $B \not \subset \osupp(D)$, then for all $\alpha$, $|\alpha| \neq 0$, $\mathcal{M} = \infty$.\\

\end{enumerate}
\end{thm}

\begin{rem}
For our experiment, the assumptions on Theorem \ref{thm_Lin_inclusionDetection} are not satisfied. In the theorem, we assume that the Lam\'{e} parameters of the inclusion are larger, while the density is lower than that of the background. In practice, as is also the case in our experiment, both, the Lam\'{e} parameters as well as the density of the inclusion is larger than that of the background. Hence, we are required to adjust the Theorem accordingly.
\end{rem}

We let from now on $D_1, D_2, D_3 \Subset \Omega$, and assume that $\lambda,\mu, \rho \in L_+^\infty(\Omega)$ 
are such that
\begin{equation} \label{eq_lambdaMuRho2}
\begin{aligned} 
\lambda &= \lambda_0 + \chi_{D_1} \psi_\lambda , \qquad \psi_\lambda \in L^\infty(\Omega),
\quad \psi_\lambda > \epsilon, \\
\mu &= \mu_0 + \chi_{D_2} \psi_\mu, \qquad \psi_\mu \in L^\infty(\Omega),\quad \psi_\mu > \epsilon, \\
\rho \,\,&{\color{black} = \rho_0 + \chi_{D_3} \psi_\rho, \qquad \psi_\rho \in L^\infty(\Omega),\quad  \psi_\rho > \epsilon}, 
\end{aligned}
\end{equation}
where the constants $\lambda_0,\mu_0,\rho_0 > 0$ and $\epsilon > 0$.
Again, the coefficients $\lambda,\mu$ and $\rho$  model inhomogeneities in an otherwise homogeneous background medium
determined by $\lambda_0,\mu_0$ and $\rho_0$.

\begin{thm}\label{thm_Lin_inclusionDetection_Mod}
Let $D := D_1 \cup D_2 \cup D_3$, where the sets are as in \eqref{eq_lambdaMuRho2} and
$B \subset \Omega$  and {\color{black}$\alpha_j \geq 0$}, 
and set $\alpha:=(\alpha_1,\alpha_2,\alpha_3)$. Let
$$
\mathcal{M}:= \# \big\{\sigma \in \spec(\Lambda_0 + \Lambda'_0[\alpha_1,\alpha_2,-\alpha_3] - \Lambda)
\;:\; \sigma < 0 \big\},
$$
where $\Lambda_0$ and $\Lambda$ are the NtD-maps for the coefficients $\lambda_0,\mu_0,\rho_0$ and $\lambda,\mu,\rho$ respectively,
and where $\Lambda'_0[\alpha_1,\alpha_2,-\alpha_3] := \Lambda'_0[\alpha_1\chi_B,\alpha_2\chi_B,-\alpha_3\chi_B]$. 
{\color{black}
Under the assumption that
\begin{align}\label{assumption}
\int_{\Omega} 2(\mu-\mu_0)\Vert \hat{\nabla}u_0\Vert^2 + (\lambda-\lambda_0)\Vert \nabla\cdot u_0\Vert^2\,dx 
> \int_{\Omega} \omega^2 (\rho - \rho_0) \Vert u_0\Vert^2\, dx,
\end{align}
}
there exists a {\color{black}$\gamma_0 > 0$} such that the following holds: \\
\begin{enumerate}

\item 
Assume that  $B \subset D_j$, for $j \in J$, for some $J \subset\{1,2,3\}$.  
Then for all $0\leq\alpha_j$ with $\alpha_j \leq {\color{black}\gamma_0}$, $j\in J\cap\{1,2\}$, and $\alpha_j = 0$, 
$j \notin J$, we have that  $\mathcal{M}< \infty$.
\\

\item
If $B \not \subset \osupp(D)$, then for all $\alpha$, $|\alpha| \neq 0$, $\mathcal{M} = \infty$.\\

\end{enumerate}
\end{thm}

\begin{rem}
In order to prove Theorem \ref{thm_Lin_inclusionDetection_Mod}, we closely follow the proof of Theorem \ref{thm_Lin_inclusionDetection} as denoted in \cite{EP24a} and adjust it for the fact that now the inclusion has a larger density as the background. For that, we need the additional Assumption (\ref{assumption}) in the Theorem. It should be noted, that the assumption itself is an assumption on the frequencies allowed to use for testing, since the chosen frequency $\omega$ can be used to control the term on the right hand side with $\omega\to 0$ leading to the stationary case discussed in \cite{EM21}.
\end{rem}

\begin{proof}
We start by proving part $(1)$ of the claim.
It will be convenient to use the abbreviations 
$$
\bar \lambda := \lambda - \lambda_0
\qquad \bar \mu := \mu - \mu_0,
{\color{black}\qquad \bar \rho := \rho - \rho_0},
$$
Now consider $w := u - u_0$. By inserting the difference into the wave equation, we see that $w$ is a weak solution 
of 
\begin{align*}  
\begin{cases}
\nabla \cdot (\C_0\,  \hat \nabla w )  + \omega^2\rho_0 w  &=-\left(  
\nabla \cdot (\C_{\bar \lambda, \bar \mu}\,  \hat \nabla u )  + \omega^2 \bar \rho u\right), \\
\;\quad\quad\quad\quad\quad(\gamma_{\C_0} w ) |_{\Gamma_N} &= ( \gamma_{ \C_{\bar \lambda, \bar \mu}} u ) |_{\Gamma_N}, \\	
 \quad\quad\quad\quad\quad\quad \quad w |_{\Gamma_D} &= 0,	
\end{cases}
\end{align*}
where $\C_0$ corresponds to $\lambda_0$ and $\mu_0$.
From the estimate in Proposition \ref{prop_elipEst_1} (see Appendix) with $A = -\C_{\bar \lambda, \bar \mu}\,  \hat \nabla u$ 
and $F = -\omega^2 \bar \rho u$, we easily see that
\begin{small}
\begin{align*}
\| w \|_{ H^1(\Omega)^3} 
\leq 
C \big( 
\| \lambda - \lambda_0\|_{L^\infty }\| \nabla \cdot u \|_{ L^2(D_1) }
+ \| \mu - \mu_0 \|_{L^\infty }\| \hat \nabla u\|_{ L^2(D_2)^{3\times 3}}
+ {\color{black}\| \rho- \rho_0 \|_{L^\infty }\| u \|_{ L^2(D_3)^3 }}
\big).
\end{align*}
\end{small}

%
\noindent
Next using this and the triangle inequality we have that
\begin{small}
\begin{align*} 
\int_{D_2}  2\bar \mu |\hat \nabla u_0 |^2\,dx 
\leq
 2 \int_{D_2}  2 \bar \mu (|\hat \nabla w |^2 +  |\hat \nabla u |^2)\,dx 
\leq C
\big( 
&\| \bar \lambda \|_{L^\infty }\| \nabla \cdot u \|^2_{ L^2(D_1) } 
+ \| \bar \mu \|_{L^\infty } \| \hat \nabla u\|^2_{ L^2(D_2)} \\
&+\| \bar \rho \|_{L^\infty }\| u \|^2_{ L^2(D_3)^3 }
\big).
\end{align*}
\end{small}
We obtain similarly, that
\begin{align*} 
\int_{D_1}  \bar \lambda |\nabla \cdot u_0 |^2\,dx 
\leq 
C \big( 
\| \bar \lambda \|_{L^\infty }\| \nabla \cdot u \|^2_{ L^2(D_1) } 
+\| \bar \mu \|_{L^\infty } \| \hat \nabla u\|^2_{ L^2(D_2)^{3\times 3}} 
+\| \bar \rho \|_{L^\infty }\| u \|^2_{ L^2(D_3)^3 }
\big).
\end{align*}
and that
\begin{align*} 
\int_{D_3} \bar \rho | u_0 |^2\,dx 
\leq C 
\big( 
\| \bar \lambda \|_{L^\infty }\| \nabla \cdot u \|^2_{ L^2(D_1) } 
+\| \bar \mu \|_{L^\infty } \| \hat \nabla u\|^2_{ L^2(D_2)^{3\times 3}} 
+\| \bar \rho \|_{L^\infty }\| u \|^2_{ L^2(D_3)^3 }
\big).
\end{align*}
{\color{black}We define $\gamma_{\mu},\gamma_{\lambda} > 0$ as
$$
\gamma_{\mu}=\min_{x \in D_2 } \frac{\bar \mu(x)}{\|\bar \mu \|_{L^\infty(D_2)}},\; 
\gamma_{\lambda}=\min_{x \in D_1 } \frac{\bar \lambda(x) }{\| \bar \lambda \|_{L^\infty(D_1) }},\;
$$
where the estimate holds because of \eqref{eq_lambdaMuRho2}.
}
Using the inequality of Lemma \ref{lem_monotonicity_ineq1} and the {\color{black}two previous inequalities}, 
and that  
$$
 \lambda - \lambda_0,\;  \mu - \mu_0 , \; \rho - \rho_0 \geq 0,  \quad \text{ in }B,
$$
we obtain 
\begin{align}
\label{eq_mono_applied}
\big(  (\Lambda_0 - \Lambda)g, \,g  \big )_{L^2(\Gamma_N)^3} 
&\geq
\int_\Omega  2(\mu-\mu_0 ) |\hat \nabla u |^2 + (\lambda - \lambda_0 ) |\nabla \cdot u|^2 + \omega^2(\rho_0-\rho) |u|^2 \,dx \nonumber \\
&\geq 
{\color{black}
\gamma_{\lambda} \| \bar \lambda \|_{L^\infty }\| \nabla \cdot u \|^2_{ L^2(D_1) } 
+ \gamma_{\mu} \| \bar \mu  \|_{L^\infty }\| \hat \nabla u\|^2_{ L^2(D_2)^{3\times 3}}
+ \int_{\Omega} \omega^2(\rho_0-\rho) |u|^2 \,dx }
\\
&\geq
{\color{black} C \left(\int_\Omega  2(\mu-\mu_0 ) |\hat \nabla u_0 |^2 + (\lambda - \lambda_0 ) |\nabla \cdot u_0|^2 \,dx - \int_\Omega \omega^2(\rho -\rho_0) |u_0|^2 \,dx\right)} 
\nonumber \\
&>0 \nonumber
\end{align}
with  some $C>0$, $g \in V^\perp$ and {\color{black}under assumption (\ref{assumption})}.
From this and \eqref{bilinear_Frechet}  we get that
\begin{align*}  
\big(  (\Lambda_0  + \Lambda_0'[\alpha_1,\alpha_2,-\alpha_3] - \Lambda)g, \,g  \big )_{L^2(\Gamma_N)^3} 
&\geq
C \int_\Omega  2(\mu-\mu_0 ) |\hat \nabla u_0 |^2 + (\lambda - \lambda_0 ) |\nabla \cdot u_0|^2  \,dx \\
& \quad  -{\color{black}C \int_\Omega \omega^2(\rho-\rho_0) |u_0|^2 \,dx }\\
&\quad - {\color{black} \int_\Omega  \alpha_2 \chi_B |\hat \nabla u_0 |^2 + \alpha_1 \chi_B |\nabla \cdot u_0|^2-\alpha_3 \chi_B | u_0|^2\,dx}. 
\end{align*}
\noindent

Since Assumption (\ref{assumption}) holds, there exists $\gamma_0$, so that 
\begin{align}  \label{eq_mono1}
\big(  (\Lambda_0  + \Lambda_0'[\alpha_1,\alpha_2,-\alpha_3] - \Lambda)g, \,g  \big )_{L^2(\Gamma_N)^3} \geq0,\qquad g\in V^\perp,
\end{align}
provided that we set $\alpha_j=0$, when $B\not\subset D_j$, and that
\begin{align*}
0\leq\alpha_j\leq\gamma_0,\qquad j=1,2\qquad\textnormal{and}\qquad0\leq\alpha_3.
\end{align*}

The inequality of \eqref{eq_mono1} implies by Lemma 6.1 in \cite{EP24} that $\mathcal{M} < \infty$.
This proves the first part of the claim.

\medskip
\noindent
Next we prove part (2) of the claim. 
We can assume that $B \cap D = \emptyset$, by considering a subset of $B$ if needed.
Assume that the claim is false and that $\mathcal{M} < \infty$, so that 
$(\Lambda_0 + \Lambda'_0[\alpha_1,\alpha_2,-\alpha_3] - \Lambda)$
has finitely many negative eigenvalues.
By Lemma 6.1 in \cite{EP24} we have a finite dimensional subspace $V_1 \subset L^2(\Gamma_N)$, such that  
\begin{equation} \label{eq_pos}
\begin{aligned}
\big((\Lambda_0 + \Lambda'_0[\alpha_1,\alpha_2,-\alpha_3] - \Lambda) g, \, g\big)_{L^2(\Gamma_N)^3}
\geq  0,  \qquad g \in V_1^\perp.
\end{aligned}
\end{equation}
We will obtain a contradiction from this.
From Lemma \ref{lem_monotonicity_ineq1} we get that
\begin{align} \label{eq_upper_part2}
\big( (\Lambda_0 + \Lambda'_0[\alpha_1,\alpha_2,-\alpha_3] - \Lambda) g, \,g  \big)_{L^2(\Gamma_N)^3} 
&\leq
\int_\Omega (\mu - \mu_0- \alpha_2 \chi_B ) | \hat \nabla u_0|^2 \,dx \nonumber \\ 
&\quad +\int_\Omega (\lambda - \lambda_0 - \alpha_1 \chi_B ) | \nabla \cdot u_0|^2 \,dx \\ 
&\quad +\int_\Omega \omega^2(\rho_0 - \rho - \alpha_3 \chi_B ) |  u_0|^2 \,dx \nonumber 
\end{align}
where $u_0$ solves \eqref{eq_bvp1} with coefficients given by  $\lambda_0, \mu_0$ and $\rho_0$, and the  
boundary condition $g \in V_2^\perp$, where $V_2$ is a finite dimensional subspace. 
The  terms on right  hand side of \eqref{eq_upper_part2} can be split as 
\begin{align} \label{eq_mu_term} 
\int_\Omega (\mu - \mu_0- \alpha_2 \chi_B ) | \hat \nabla  u_0|^2 \,dx  
=
\int_{D_2} (\mu - \mu_0) | \hat \nabla  u_0|^2 \,dx  
- \int_{B} \alpha_2 \chi_B | \hat \nabla  u_0|^2 \,dx 
\end{align}
and 
\begin{align} \label{eq_lambda_term} 
\int_\Omega (\lambda - \lambda_0- \alpha_1 \chi_B ) | \nabla  \cdot u_0|^2 \,dx  
=
\int_{D_1} (\lambda - \lambda_0) |  \nabla \cdot u_0|^2 \,dx  
- \int_{B} \alpha_1 \chi_B | \hat \nabla  u_0|^2 \,dx 
\end{align}
and 
\begin{align} \label{eq_rho_term} 
{\color{black}\int_\Omega \omega^2(\rho_0 - \rho - \alpha_3 \chi_B ) |  u_0|^2 \,dx 
=
-\int_{D_3} \omega^2(\rho - \rho_0 ) |  u_0|^2 \,dx 
- \int_{B} \omega^2\alpha_3 \chi_B  |  u_0|^2 \,dx }
\end{align}
We now use localized solutions that become large in the set $B$ and small in $D$.
By choosing the sets $D'_1$ and $D'_2$ in Proposition 4.1 in \cite{EP24a}, as  $D'_1 = \osupp(D)$ and $D'_2 = B$,
and suitable sets $D_1$ and $D_2$,
we get a sequence $g_j  = (\gamma_\C u_{0,j})|_{\p\Omega} \in (V_1 \cup V_2 )^\perp$ of
boundary data that give the  solutions $u_{0,j}$ to \eqref{eq_bvp1}, with the coefficients
$\lambda_0, \mu_0$ and $\rho_0$, such that
$$
\| u_{0,j} \|_{ L^2(\osupp(D))^3 }, 
\;\| \hat \nabla u_{0,j} \|_{ L^2(\osupp(D))^{3\times 3} }, 
\;\| \nabla \cdot u_{0,j} \|_{ L^2(\osupp(D))} 
\to 0, 
$$
and such that
$$
\| u_{0,j} \|_{ L^2(B)^3 }, 
\;\| \hat \nabla u_{0,j} \|_{ L^2(B)^{3\times 3} }, 
\;\| \nabla \cdot u_{0,j} \|_{ L^2(B)} 
\to \infty, 
$$
as $j \to \infty$.
If $|\alpha| \neq 0$, then it  follows from these limits and equations \eqref{eq_mu_term}, 
\eqref{eq_lambda_term} and \eqref{eq_rho_term}, that \eqref{eq_upper_part2} gives the estimate 
\begin{align*} 
\big( (\Lambda_0 + \Lambda'_0[\alpha_1,\alpha_2,\alpha_3] - \Lambda) g_j, \,g_j  \big)_{L^2(\Gamma_N)^3} < 0, 
\end{align*}
when $j$ is large enough, with $g_j \in (V_1 \cup V_2 )^\perp \subset V_1 ^\perp$. 
But this is in contradiction with \eqref{eq_pos}. Part $(2)$ of the claim thus holds.

\end{proof}

For practical purposes, we cannot test for infinite eigenvalues. Hence, we slightly modify Theorem \ref{thm_Lin_inclusionDetection_Mod} according to Lemma 3.4 from \cite{E25}.

\begin{lem}\label{lem_Lin_inclusionDetection}
Let $D := D_1 \cup D_2 \cup D_3$, where the sets are as in \eqref{eq_lambdaMuRho2} and
$B \subset \Omega$  and {\color{black}$\alpha_j \geq 0$}, 
and set $\alpha:=(\alpha_1,\alpha_2,\alpha_3)$.
Let
$$
\mathcal{M}_k:= \# \big\{\sigma \in \spec(\Lambda_0 + \Lambda'_0[\alpha_1\chi_{B_k},\alpha_2\chi_{B_k},-\alpha_3\chi_{B_k}] - \Lambda)
\;:\; \sigma < 0 \big\},
$$
where $\Lambda_0$ and $\Lambda$ are the NtD-maps for the coefficients $\lambda_0,\mu_0,\rho_0$ and $\lambda,\mu,\rho$ respectively. 
Let further $\mathcal{B}=\left\{B_k\ |\ B_k\subset \Omega\right\}$ be a fixed finite set. 
Under the assumption that
\begin{align}
\int_{\Omega} 2(\mu-\mu_0)\Vert \hat{\nabla}u_0\Vert^2 + (\lambda-\lambda_0)\Vert \nabla\cdot u_0\Vert^2\,dx 
> \int_{\Omega} \omega^2 (\rho - \rho_0) \Vert u_0\Vert^2\, dx,
\end{align}
there exists a $\gamma_0 > 0$ and a $0\leq M_l$ such that the following holds: \\

 \begin{enumerate}

\item 
Assume that  $B_k \subset D_j$, for $j \in J$, for some $J \subset\{1,2,3\}$.  
Then for all $0\leq\alpha_j$ with $\alpha_j \leq \gamma_0$, $j \in J\cap\{1,2\}$, and $\alpha_j = 0$, 
$j \notin J$, we have that  $\mathcal{M}_k\leq M_l$.
\\

\item
If $B_k \not \subset \osupp(D)$, then for all $\alpha$, $|\alpha| \neq 0$, $\mathcal{M}_k > M_l$.

\end{enumerate}
\end{lem}

\noindent
\\
Since there are various sources of error in the data collection of a laboratory experiment, our mathematical methods must also be applicable to this noisy data. Possible sources of errors that must be considered include measurement errors (tolerances of the measuring instruments), stochastic errors, and numerical errors.
\\
\\
Thus, we continue with the consideration of noisy data. 

\section{Monotonicity methods for noisy data}
\label{noise_teil}

Next, we formulate the test for noisy data similar to Theorem 4.2 in \cite{E25}. The proof can be adopted as well.

 \begin{thm} \label{thm_Lin_inclusionDetection_noisy}
Let $D := D_1 \cup D_2 \cup D_3$, where the sets are as in \eqref{eq_lambdaMuRho} and let
$\mathcal{B}=\left\{B_k\ |\ B_k \subset \Omega\right\}$ be fixed and finite. Further, let $\alpha_j > 0$, 
and set $\alpha:=(\alpha_1,\alpha_2,\alpha_3)$. Let
$$
\mathcal{M}_k:= \# \big\{\sigma \in \spec(\Lambda_0 + \Lambda'_0[\alpha_1\chi_{B_k},\alpha_2\chi_{B_k},-\alpha_3\chi_{B_k}] - \Lambda)
\;:\; \sigma < -\delta \big\},
$$
where $\Lambda_0$ and $\Lambda$ are the NtD-maps for the coefficients $\lambda_0,\mu_0,\rho_0$ and $\lambda,\mu,\rho$ respectively. Further, let 
 \begin{align*}
 \Vert \Lambda^\delta(\lambda,\mu,\rho) - \Lambda(\lambda,\mu,\rho)\Vert < \delta.
 \end{align*}
 \noindent
Under the assumption that
\begin{align}
\int_{\Omega} 2(\mu-\mu_0)\Vert \hat{\nabla}u_0\Vert^2 + (\lambda-\lambda_0)\Vert \nabla\cdot u_0\Vert^2\,dx 
> \int_{\Omega} \omega^2 (\rho - \rho_0) \Vert u_0\Vert^2\, dx,
\end{align}
there exists a {\color{black}$\gamma_0 > 0$}, a $0\leq M_l$ and a maximal noise level $\delta_0>0$, such that for all $0 < \delta <\delta_0$ we have the following statements: \\
\begin{enumerate}

\item 
Assume that  $B_k \subset D_j$, for $j \in J$, for some $J \subset\{1,2,3\}$.  
Then for all $0\leq\alpha_j$ with $\alpha_j \leq  {\color{black}\gamma_0}$, $j \in J\cap\{1,2\}$, and $\alpha_j = 0$, 
$j \notin J$, we have that  $\mathcal{M}_k \leq M_l$.
\\

\item
If $B_k \not \subset \osupp(D)$, then for all $\alpha$, $|\alpha| \neq 0$, $\mathcal{M}_k >M_l$.\\


\end{enumerate}
\end{thm}

\noindent
\\
In order to close this section, we formulate the corresponding algorithm:
\begin{algorithm} 
\caption{Linearized reconstruction of  $\osupp(D) \subset \Omega$.}\label{alg_shapeInclusion}
\begin{algorithmic}[1]

\STATE  Choose a fixed finiteset $\mathcal{B} = \{ B_k\,| B_k\subset\Omega\}$ and set $\mathcal{A} = \{\}$.

\STATE  Choose $\tilde{M}_l$.


\FOR{  $B \in \mathcal{B}$}

\FOR{  $\Lambda^\flat$ with parameters varied as suggested by Theorem \ref{thm_Lin_inclusionDetection_noisy}}

\STATE Compute $\displaystyle\mathcal{M}_B := \sum_{\sigma_k < -\delta} 1$, where $\sigma_k$ are the eigenvalues of 
\STATE $\Lambda_0 + \Lambda'_0[\alpha_1,\alpha_2,-\alpha_3] - \Lambda^\delta$ 

\IF {$ \mathcal{M}_B < \tilde{M}_l$}

\STATE 
Add $B$ to the approximating collection $\mathcal{A}$, since 
Theorem \ref{thm_Lin_inclusionDetection_noisy} suggests 
\STATE 
that $B \subset \osupp(D)$. 

\ELSE

\STATE Discard $B$, since by Theorem \ref{thm_Lin_inclusionDetection_noisy} $B \not \subset D_j$, $j=1,2,3$.

\ENDIF
\ENDFOR
\ENDFOR

\STATE  Compute the union of all elements in $\mathcal{A}$ and all components of 
$\Omega \setminus \cup \mathcal{A}$ not connected 
\STATE to $\p\Omega$. The resulting set is an approximation of $\osupp(D)$.

\end{algorithmic}
\end{algorithm}
\noindent
\\
We want to remark, that we choose $\tilde{M}_l$ as described in \cite{E25}. In more detail, we take a look at the negative eigenvalues for all considered test inclusions and introduce $\tilde{M}_l$ such that we mark the test inclusions with a (significant) lower number of negative eigenvalues than the rest of the test inclusions.

\section{Inversion of experimental data} \label{sec_numerics}
The material parameters of the elastic body are given in Table \ref{Table_parameter}.

\renewcommand{\arraystretch}{1.4}
\begin{table}[h!]
\begin{center}
\begin{tabular}{ |c|c| c | c |}  
\hline
material & $\lambda_i$ in $[Pa]$ & $\mu_i$ in $[Pa]$  & $\rho_i$ in $[kg/m^3]$ \\
\hline
$i=0$: background (Makrolon)&  $2.8910\cdot 10^9$   &  $1.1808\cdot10^9$  &  $1171$\\
 \hline
$i=1$: inclusion (aluminum)&  $5.1084\cdot 10^{10}$ &  $2.6316\cdot 10^{10}$  &   $2700$\\
\hline
\end{tabular}
\end{center}
\caption{Lam\'e parameter and density for the experimental setup.}
\label{Table_parameter}
\end{table}
\noindent
\\
The experiment is conducted by applying a sine sweep covering the frequency bands from $20Hz -27Hz$, $40Hz-45Hz$ and $55Hz-57Hz$. The frequencies and especially the gaps in the frequency bands are chosen in such a way that unwanted strong vibrations of the test objects are avoided in order to reduce expected measurement noise. For the chosen frequencies, Assumption (\ref{assumption}) clearly holds. 
More explicitly, the left hand side of the inequality is of order $10^{-16}$, while the right hand side is of order $10^{-19}$. According to our numerical simulations, the Assumption (\ref{assumption}) holds up to frequencies in the range up to $100$ Hz due to the Lam\' parameters given in $GPa$, i.e., $10^9\frac{kg m}{s^2m}$, while the density is given in $\frac{kg}{m^3}$.
\\
\begin{figure}[H]
\centering 
\includegraphics[width=1\textwidth]{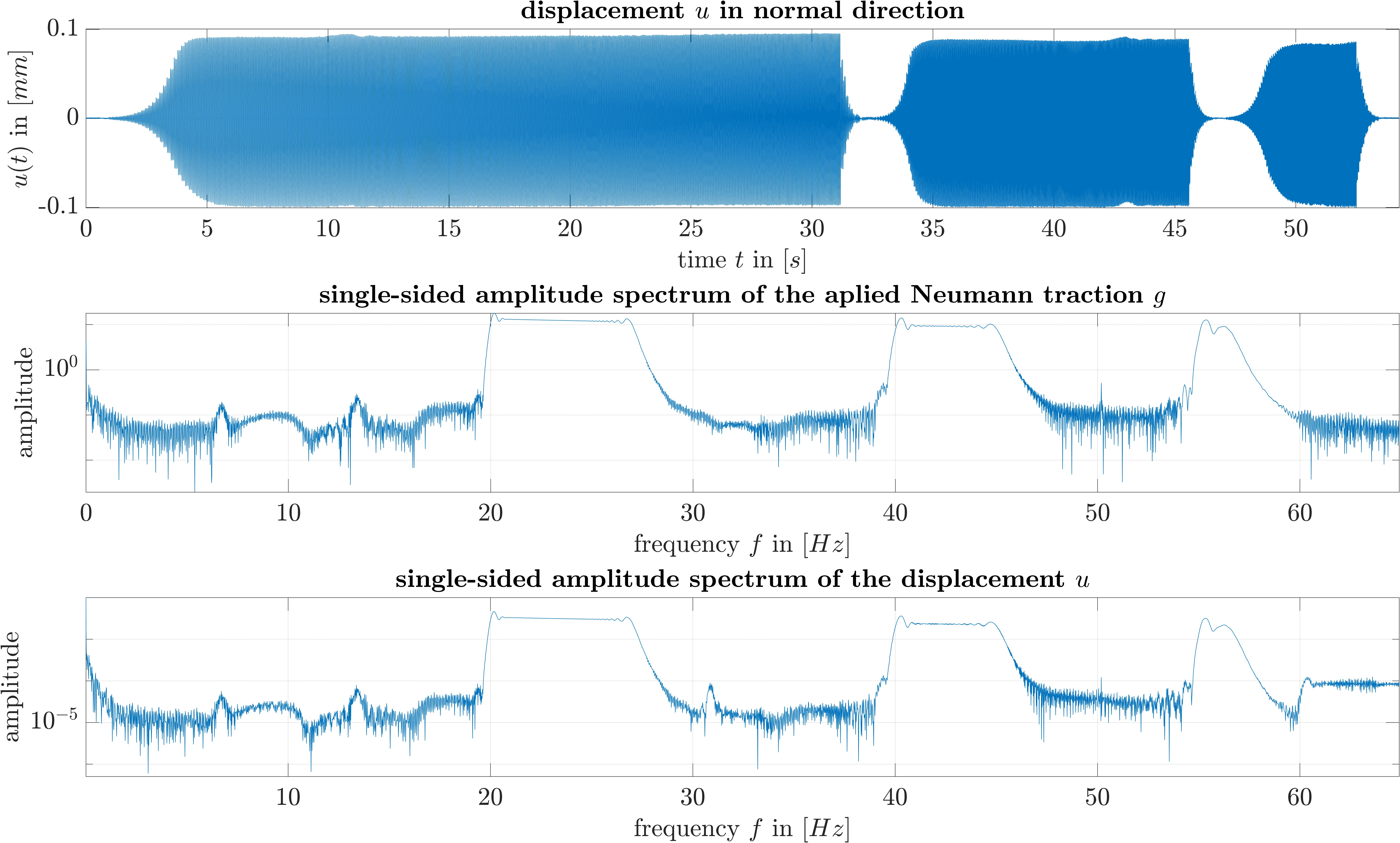}
\caption{Sine sweep applied to each Neumann boundary}
\end{figure}
\noindent
The displacements on the Neumann boundary were then measured and the applied boundary forces were recorded by the shaker. Both data sets were then made consistent by e.g. cutting of data recorded before and after the sweep. At last, both the applied force and the measured displacements were Fourier transformed into the frequency domain for a fixed frequency, where the monotonicity tests as described in section \ref{noise_teil} are applied. For that, missing data, which could not be measured due to the experimental setup were interpolated according to the spline interpolation method described in \cite{EM21}, where a similar experiment based on the stationary wave equation was performed and described. 
\\
The reconstruction was conducted with 25 test inclusions ($5\times5$) of dimension $6cm\times6cm\times 1cm$. A reconstruction using a finer grid was not possible due to the inherent noise in the data caused by vibrations.
\\
\\
The plots are organized as follows: The illustration shows the plate with the test inclusions, the unknown inclusions, the points where the force is applied or the displacement is measured, and the position of the plate holder. The green marking of a test inclusion means that it was detected as lying outside the unknown inclusions, and red means that a test inclusion was detected as lying within the unknown inclusions.

\subsection*{Centralized inclusion with 12cm diameter}

The reconstruction of the centralized aluminum inclusion with a diameter of $12cm$ was consistent in all three frequency bands and can be seen in Figure \ref{Fig_12cm}. Exemplary, this result was obtained for the parameter sets mentioned in Table \ref{tab_parameter_12cm}.
\\
\\
It should be noted that the experiment is heavily burdened by noise due to its setup. The Makrolon plate is only fixed in place by clamping the plate in two points in the middle of two opposing sides, so that most of the plate is free-floating. Further, we excite the plate with oscillations close to as well as far from the Dirichlet boundary. Nevertheless, we obtain a consistent reconstruction of the inclusion over all frequency bands. In a further application, the setup should be adapted to reduce unwanted noise as much as possible by, e.g., clamping the plate in its four corners and provide more robustness to the algorithm by taking more measurements.

\begin{figure}[H]
\centering 
\includegraphics[width=0.4\textwidth]{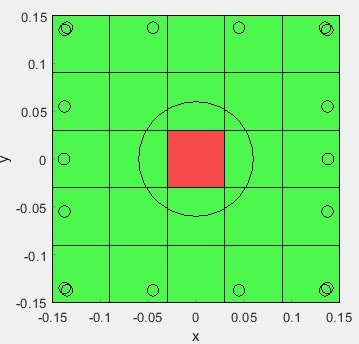}
\caption{Detected inclusions of the experiment with a 12cm central inclusion. The inclusion is marked in red.}\label{Fig_12cm}
\end{figure}

\begin{table*}[h]
	\centering
		\begin{tabular}{|c|c|c|}\hline
		$\omega$ & $M$ & $\delta$\\\hline
		21$\frac{rad}{s}$ &6& $9.775038\cdot 10^{-7}$\\\hline
		41$\frac{rad}{s}$ &6& $9.7283458500\cdot 10^{-7}$\\\hline
		55.4$\frac{rad}{s}$ &6& $7.929299\cdot 10^{-7}$\\\hline
		\end{tabular}
				\\
		\vspace{0.5cm}
		\caption{
		Parameter combinations resulting in the reconstruction presented in Figure \ref{Fig_12cm}.}
		\label{tab_parameter_12cm}
\end{table*}

\subsection*{Two decentralized inclusions with 10cm diameter}

As a further and harder test sample, we tested a Makrolon plate with two inclusions with a diameter of $10cm$. The inclusions are positioned in opposite corners. Due to the position of the Dirichlet boundary (in the middle of two opposing sides) and the position of the heavier aluminium inclusions in the corners, we could detect very strong fluttering of the corners of the test object perpendicular to the direction of excitation. The higher the oscillation frequency was applied to the boundary, the more fluttering was observed. This was not the case in the first test, where the inclusion is centralized exactly between the Dirichlet boundaries. As expected, it was not possible to reconstruction the inclusion in the wave bands $40Hz-45Hz$ and $55Hz-57Hz$.
\\
\\
However in the lower wave band $20Hz -27Hz$, we were able to reconstruct and separate the inclusion correctly (see Figure \ref{Fig_2x10cm}) for low frequencies. It should be noted that the algorithm can only make correct statements about test inclusions which lie completely inside or outside the inclusion to be reconstructed. In order to obtain a more precise approximation of the boundary of the inclusion to be detected, finer test inclusions are necessary, which was not possible for this experiment as already mentioned earlier. For the reconstruction with higher frequencies, more stability has to be introduced to the experiment (e.g., clamping in the corners), which was not possible for this experiment due to budget and time constraints.

\begin{figure}[H]
\centering 
\includegraphics[width=0.4\textwidth]{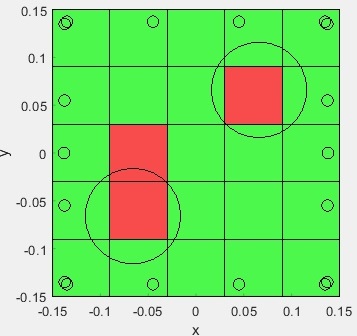}
\caption{Detected inclusions of the experiment wit a two separate $10cm$ decentralized inclusion. The reconstructed inclusion is marked in red. Parameters: $\omega=20.2 \frac{rad}{s}$, $M=6$, $\delta=1.53598375\cdot 10^{-6}$}
\label{Fig_2x10cm}
\end{figure}
\noindent
Finally, we want to remark that compared with the results for the stationary case (see \cite{EM21}), we obtain better results especially for the two decentralized inclusions.

\bigskip
\noindent
{\bf{Acknowledgement}} \\
The first author thanks the German Research Foundation
(DFG) for funding the project ”Inclusion Reconstruction with Monotonicity-based
Methods for the Elasto-oscillatory Wave Equation” (reference number 499303971)
at the Goethe-University Frankfurt, where the major part of this article has been
conducted during this project. 
\\
We would like to thank Johannes Käsgen for his support in planning and carrying out the experiment at the Fraunhofer Institute for Structural Durability and System Reliability LBF.

\section{Appendix}
\noindent
Here we summaries the required background from \cite{EP24a} which is used for the well-posedness  and apriori estimates, for the weak solutions 
to the source problem
$$
u \in \mathcal{V} :=  \{ u \in H^1(\Omega)^3 \,:\, u|_{\Gamma_D} = 0 \}
$$
to the formal boundary value problem
\begin{align}  \label{eq_bvp2}
\begin{cases}
\nabla \cdot (\C\,  \hat \nabla u )  + \omega^2\rho u + \tau u &=  F + \nabla \cdot A \\
\;\quad\quad\quad\quad(\gamma_\C u ) |_{\Gamma_N} &=  A \nu |_{\Gamma_N} , \\	
 \quad\quad\quad\quad\quad \quad u |_{\Gamma_D} &= 0,	
\end{cases}
\end{align}
where $F \in L^2(\Omega)^3$ and $A \in L^2( \Omega)^{3 \times 3}$. 

\begin{prop} [see Proposition 8.1 from \cite{EP24}]\label{prop_elipEst_1}
There exists a $\tau_0 \leq 0$, for which
the boundary value problem in \eqref{eq_bvp2} admits a unique weak solution
$u \in \mathcal{V}$, which satisfies
\begin{align}  \label{eq_aprioriEst}
\| u \|_{ H^1(\Omega)^3} 
\leq C 
(\| F \|_{L^2( \Omega)^{3}}+\| A \|_{L^2( \Omega)^{3 \times 3}}).
\end{align}
\end{prop}

\noindent
\\
\\
\end{document}